\documentclass[12pt]{amsart}

\usepackage[all]{xy}
\usepackage{amscd,amsmath,amsthm,amssymb}
\usepackage{mathtools}
\usepackage{mathrsfs}
\usepackage{tikz}

\definecolor{cadmiumgreen}{rgb}{0.0, 0.42, 0.24}
\usepackage[
colorlinks, citecolor=cadmiumgreen,
pdfauthor={Robin de Jong, Farbod Shokrieh}, 
pdfstartview ={FitV},
]{hyperref}
\hypersetup{
pdftitle={Canonical local heights and Berkovich skeleta},
}

\usepackage[
msc-links,
nobysame,
lite,
]{amsrefs} 

\usepackage[normalem]{ulem}

\usepackage[left=3.5cm,top=3.5cm,right=3.5cm]{geometry}
\usepackage{microtype}
\mathtoolsset{showonlyrefs}

\newtheorem{thm}{Theorem}[section]
\newtheorem{prop}[thm]{Proposition}
\newtheorem{cor}[thm]{Corollary}

\theoremstyle{remark}
\newtheorem{remark}[thm]{Remark}
\newtheorem{example}[thm]{Example}

\theoremstyle{definition}
\newtheorem{definition}[thm]{Definition}

\makeatletter
\renewcommand*\env@matrix[1][*\c@MaxMatrixCols c]{
  \hskip -\arraycolsep
  \let\@ifnextchar\new@ifnextchar
  \array{#1}}
\newcommand*\isom{\xrightarrow{\sim}}
\newcommand{\divisor}{\operatorname{div}}
\newcommand{\ord}{\operatorname{ord}}

\newcommand{\Hom}{\operatorname{Hom}}

\newcommand{\Spec}{\operatorname{Spec}}

\newcommand{\Sk}{\operatorname{Sk}}

\def\opn#1#2{\def#1{\operatorname{#2}}}
\opn\Vor{Vor}

\def\qq{\mathbb{Q}}

\def\rr{\mathbb{R}}

\def\zz{\mathbb{Z}}
\def\cc{\mathcal{C}}
\def\gg{\mathbb{G}}

\def\nn{\mathcal{N}}
\def\mm{\mathcal{M}}
\def\ll{\mathcal{L}}
\def\bb{\mathcal{B}}
\def\pp{\mathcal{P}}

\def\tt{\mathcal{T}}

\def\aa{\mathcal{A}}
\def\oo{\mathcal{O}}

\def\dd{\mathcal{D}}
\def\ee{\mathcal{E}}

\def\hh{\mathcal{H}}

\def\d{\mathrm{d}}

\def\for{\mathrm{for}}

\def\id{\mathrm{id}}

\def\an{\mathrm{an}}

\def\val{\mathrm{val}}
\def\sp{\mathrm{sp}}

\def\t{\mathrm{t}}

\def\h{\mathrm{h}}
\def\m{\mathrm{m}}

\def\trop{\mathrm{trop}}

\numberwithin{equation}{section}

\subjclass[2020]{\href{https://mathscinet.ams.org/msc/msc2020.html?t=11G10}{11G10},
\href{https://mathscinet.ams.org/msc/msc2020.html?t=11G50}{11G50},
\href{https://mathscinet.ams.org/msc/msc2020.html?t=14G40}{14G40},
\href{https://mathscinet.ams.org/msc/msc2020.html?t=14K25}{14K25},
\href{https://mathscinet.ams.org/msc/msc2020.html?t=14T25}{14T25},
\href{https://mathscinet.ams.org/msc/msc2020.html?t=14G20}{14G20},
\href{https://mathscinet.ams.org/msc/msc2020.html?t=14K05}{14K05}}

\begin{document}

\title[Canonical local heights and Berkovich skeleta]{Canonical local heights and Berkovich skeleta}

\author{Robin de Jong}
\address{Leiden University \\ Einsteinweg 55  \\ 2333 CC Leiden  \\ The Netherlands }
\email{\href{mailto:rdejong@math.leidenuniv.nl}{rdejong@math.leidenuniv.nl}}

\author{Farbod Shokrieh}
\address{University of Washington \\ Box 354350 \\ Seattle, WA 98195 \\ USA }
\email{\href{mailto:farbod@uw.edu}{farbod@uw.edu}}

\begin{abstract} 
 We discuss canonical local heights on abelian varieties over non-archimedean fields from the point of view of Berkovich analytic spaces. Our main result is a refinement of N\'eron's classical result relating canonical local heights with intersection multiplicities on the N\'eron model.
We also revisit Tate's explicit formulas for N\'eron's canonical local heights on elliptic curves. Our results can be viewed as extensions of Tate's formulas to higher dimensions.
 \end{abstract}

\maketitle

\setcounter{tocdepth}{1}
\tableofcontents

\thispagestyle{empty}

\section{Introduction} \label{sec:intro}

\renewcommand*{\thethm}{\Alph{thm}}

\subsection{Background}
In 1965, N\'eron \cite{neron} proved the existence of functorial local height functions for divisors on abelian varieties over global fields (see also \cite[Chapter~11]{lang} or \cite[\S 9.5]{bg_heights}). In his proof, he made use of the ``minimal'' models of abelian varieties that he found a year earlier, and are nowadays named after him. N\'eron's canonical local height functions are uniquely determined up to additive constants.

Let $A$ be an abelian variety over a global field $K$, let $v$ be a place of $K$, and let $D$ be a divisor on $A$. We denote by
\[ \lambda_{D,v} \colon A(K_v) \setminus |D| \to \rr \]
any N\'eron's canonical local height function for $D$, where $|D|$ denotes the support of~$D$. 

Assume that the place $v$ is \emph{non-archimedean}. We let $\nn_v$ denote the N\'eron model of $A$ over $v$, and $\varPhi_v$ the group of connected components of the special fiber of $\nn_v$. 
One of N\'eron's results (\cite[\S III.4, Th\'eor\`eme~1, p.~319]{neron}, see also \cite[\S11.5]{lang}) is that there is a map $\gamma_v \colon \varPhi_v \to \rr$ such that for all $x \in A(K_v) \setminus |D|$ the equality
\begin{equation} \label{eqn:def_gamma}
\lambda_{D,v}(x) = \mathbf{i}_v(x,D) + \gamma_v (x_0) 
 \end{equation}
holds in $\rr$. Here $x_0$ denotes the element of $\varPhi_v$ to which $x$ specializes, and $\mathbf{i}_v(x,D) $ denotes the intersection multiplicity of $\overline{x}$, the Zariski closure of $x$ in $\nn_v$, with the divisor $\overline{D}$, the \emph{thickening} of $D$ in $\nn_v$, on the regular scheme $\nn_v$.

N\'eron also showed that it is possible to choose the function $\lambda_{D,v}$ in such a way that the map $\gamma_v$ takes values in $\frac{1}{2N}\zz$, where $N \in \zz_{>0}$ is any positive integer such that $N \cdot \varPhi_v =0$.
It is customary to simply pick $\lambda_{D,v}(x) = \mathbf{i}_v(x,D)$ when $v$ is a place of good reduction of $A$. With this choice made, the canonical local  heights add up to give a global height on $A(K)$ associated to the line bundle determined by~$D$. This global height equals the canonical (N\'eron--Tate) global height, up to an additive constant.

In a 1968 letter to Serre, Tate gave explicit formulas for N\'eron's canonical local height functions in the case of \emph{elliptic curves} (see \cite[Chapter 3, \S5]{langElliptic}, \cite[\S6.5]{lectures_mordell}, \cite[Chapter VI]{sil_advanced}), and indicated a way of normalizing them such that the resulting global height exactly equals the canonical height.

\subsection{Our contribution}
We study N\'eron's canonical local heights on abelian varieties from the point of view of \emph{Berkovich analytic spaces}. With this point of view, we are able to formulate a refinement of N\'eron's result (see Theorem~\ref{thm:A}).
As we will see, the choice $\lambda_{D,v}(x) = \mathbf{i}_v(x,D)$, when $v$ is a place of good reduction of $A$, is not merely ad hoc; we will define the notion of a \emph{normalized} canonical local height, and we will see that it equals $\mathbf{i}_v(x,D)$ in the good reduction case (Corollary~\ref{cor:normalized_good}). 

Our results can be viewed as extensions of Tate's formulas to higher dimensions. In fact, in \S\ref{sec:elliptic}, we revisit Tate's formulas from our point of view. 
It turns out that, for elliptic curves, our normalized canonical local heights coincide with the classical normalized local heights introduced by Tate. 

One of our tools is \emph{non-archimedean theta functions} and \emph{their tropicalizations}, as studied in \cite{frss}. Another tool in our proof is a special kind of Mumford models of abelian varieties that we call \emph{tautological models}. We study some of their properties, and make explicit the connection with the models mentioned and studied by Alexeev and Nakamura in \cite{an}. 

A fundamental insight, that we use critically in this paper, is that the local height functions $\lambda_{D,v}$ as constructed by N\'eron at the non-archimedean place $v$, extend canonically as \emph{Green's functions} on the \emph{Berkovich analytification} of $A$ at~$v$; see, e.g., \cite[\S1.3]{cl_overview}, \cite[\S2]{clt}.

We hope that our presentation can help make non-archimedean canonical local heights algorithmically accessible. In particular, in the case that $D$ corresponds to a \emph{principal polarization}, our formulas are already quite explicit (see Theorem~\ref{thm:B2}).

Using Theorem~\ref{thm:A}, we can write down a precise \emph{local-to-global formula} for the canonical global height (see Theorem~\ref{thm:B}), in which no unspecified additive constants appear. Further applications of our point of view will be presented in future work.

\subsection{Skeleta and theta functions}

Let $R$ be a complete discrete valuation ring, with fraction field $F$. We assume the absolute value on $F$ is normalized so that a uniformizer has valuation one and, hence, the value group is identified with $\zz$. Let $A$ be an abelian variety over $F$, with split semi-abelian reduction over $R$. Let $L$ be an ample (and, in particular, effective) line bundle on $A$. 

We will freely use terminology related to Berkovich analytic spaces and metrized line bundles on them as developed in, e.g., \cites{be, cl_overview, cl, clt, gu, gu_trop}.
Denote by $A^\an$ the Berkovich analytification of $A$.  Throughout, all analytifications will be over the completion of an algebraic closure of $F$. 

Raynaud's theory of non-archimedean uniformization (discussed in the context of Berkovich analytic spaces in, e.g., \cite[\S6.5]{be}) gives rise to finitely generated free abelian groups $X, Y$ of the same rank together with an injective homomorphism $Y \to X^*$, where $X^*=\Hom(X,\zz)$. The real torus $\varSigma = X_\rr^*/Y$ turns out to be naturally a subset of $A^\an$ and is called the \emph{canonical skeleton} of $A^\an$. It comes with a natural \emph{tropicalization map} $\overline{\val} \colon A^\an \to \varSigma$. It has the property that for $x \in A(F)$ we have $\overline{\val}(x) \in X^*/Y$.

Let $D$ be an effective divisor representing $L$, and let $s \in H^0(A,L)$ be a global section whose divisor is $D$. 
After equipping $L$ with a rigidification, we obtain (non-canonically) a \emph{non-archimedean theta function} $f$ associated to our given section $s$. Following \cite[\S 4]{frss}, we have associated to $f$ its \emph{tropicalization}, which is a continuous concave piecewise integral affine function $f_\trop \colon X_\rr^* \to \rr$. In \S\ref{sec:taut_metric}, we will define a normalized version $\| f_\trop \| \colon X_\rr^* \to \rr$ of $f_\trop$,  which descends along the $Y$-action on $X_\rr^*$ to give a map
\[ \| f_\trop \| \colon \varSigma \to \rr \, . \]
The map $\| f_\trop \|$ turns out to be canonically attached to the section $s$.

\subsection{Main results}

Let $\nn$ be the N\'eron model of $A$ over $R$.
Let $\varPhi_\nn$ denote the group of connected components of the special fiber  of $\nn$. We have a \emph{specialization map} $\sp \colon A(F) \to \varPhi_\nn$, which is surjective.

Denote by $L^\an$ the Berkovich analytification of the ample line bundle~$L$. The rigidification of $L$ and the group structure of $A$ induce a \emph{canonical metric} on $L^\an$, denoted $\|\cdot\|_L$ (see, e.g., \cite[Example~3.7]{gu}). Then, the function 
\[ -\log \|s\|_L \colon A^\an \setminus |D^\an| \to \rr \]
defines a {\em Green's function} with respect to the divisor $D$ (see, e.g., \cite[\S1.3]{cl_overview}, \cite[\S2]{clt}). The restriction of $-\log \|s\|_L $ to $A(F) \setminus |D|$ is a N\'eron's canonical local height with respect to $D$. This follows from the discussion in \cite[\S III.1]{mb} or in \cite[\S 9.5]{bg_heights}.
 
 Let $\sp \colon A(F) \to \varPhi_\nn$ denote the specialization map.
By \cite[Corollary~III.8.2]{fc}, the map $\varPhi_\nn \to X^*/Y$ sending $\sp(x)$, for $x \in A(F)$, to $\overline{\val}(x)$ is an isomorphism of groups. Therefore, we view the group $\varPhi_\nn$ canonically as a subgroup of the canonical skeleton $\varSigma=X_\rr^*/Y$ of $A^\an$.
 
Recall that we have fixed a non-zero global section $s \in H^0(A,L)$ whose divisor is~$D$. Let $f$ be a theta function for $s$ and let $\|f_\trop\| \colon \varSigma \to \rr$ be the associated normalized tropicalized theta function. 

\begin{thm} \label{thm:A} Let $N \in \zz_{>0}$ be any positive integer such that $N \cdot \varPhi_\nn =0$.
\begin{itemize} 
\item[(i)] The restriction of the function $\|f_\trop\| $ to $\varPhi_\nn$ takes values in $ \frac{1}{2N}\zz$.
\item[(ii)] Let $x \in A(F) \setminus |D|$ and write $x_0=\sp(x)$. The equality
\begin{equation} \label{Neron}
 -\log \|s(x)\|_L = \mathbf{i}(x,D) + \|f_\trop\| (x_0) 
 \end{equation}
holds. Here $\mathbf{i}(x,D)$ denotes the intersection multiplicity of the Zariski closure of $x$ in $\nn$ with the thickening of $D$ over $\nn$. 
\end{itemize}
\end{thm}
Every divisor on $A$ is a difference of two ample divisors. Therefore, Theorem~\ref{thm:A} gives a  formula for the canonical local height associated to \emph{any} divisor on $A$.

In the case that $L$ defines a \emph{principal polarization}, we can make the function $\|f_\trop\|$ quite explicit. In this case, we have an intimate connection with the \emph{tropical Riemann theta function} $\varPsi \colon X_\rr^* \to \rr$ associated to $L$ and its normalized version $\|\varPsi\| \colon \varSigma \to \rr$, which has a very simple description in terms of the Euclidean structure on $X_\rr^*$ determined by $L$ (see \S\ref{sec:pp} for details).  

\begin{thm}  \label{thm:B2} Assume that $L$ defines a principal polarization. There exists a unique element $\kappa \in \varSigma$ and a unique $r \in \qq$ such that $2\kappa \in X^*/Y$ and 
\[\|f_\trop\| = \|\varPsi\| \circ \t_\kappa + r\] 
on $\varSigma$. Here $\t_\kappa \colon \varSigma \to \varSigma$ denotes translation by $\kappa$. 
\end{thm}

We call the element $\kappa \in \varSigma$ in Theorem~\ref{thm:B2} the \emph{tropical theta characteristic} of~$L$.

\subsection{Earlier work} 

The formula in \eqref{eqn:def_gamma} due to N\'eron had been refined by Moret-Bailly in his monograph \cite{mb} as follows.
As above, let $s \in H^0(A,L)$ be a global section of $L$ whose divisor is~$D$. In \cite[\S III.1]{mb} one finds a construction of a N\'eron's canonical local height function $\langle s, x \rangle $ for $x \in A(F) \setminus |D|$. It is not too hard to see that, in fact, the equality
\[ \langle s, x \rangle = -\log \|s(x)\|_L  \]
holds true for $x \in A(F) \setminus |D|$. Let $N \in \zz_{>0}$ be any positive integer such that $N \cdot \varPhi_\nn =0$. The arguments in  \cite[\S III.1.3--1.4]{mb} prove the following: the rigidified line bundle $L^{\otimes 2N}$ extends uniquely as a \emph{cubical} line bundle $\overline{L^{\otimes 2N}}$ over the N\'eron model~$\nn$. Moreover, for all $x \in A(F) \setminus |D|$, we have the equality
\[  \langle s, x \rangle = \mathbf{i}(x,D) +  \frac{1}{2N} \ord_{V,\overline{L^{\otimes 2N}}}(s^{\otimes 2N}) \, , \]
where $V$ is the component of the special fiber of~$\nn$ to which $x$ specializes. Here, we view $s^{\otimes 2N}$ as a rational section of the line bundle $\overline{L^{\otimes 2N}}$ over $\nn$. 

Combining with Theorem~\ref{thm:A}(ii) we conclude 
\[ \|f_\trop\|(V) =  \frac{1}{2N} \ord_{V,\overline{L^{\otimes 2N}}}(s^{\otimes 2N}) \, , \]
where on the left hand side $V$ is thought of as an element of $\varPhi_\nn$.  This equality gives an alternative explanation of the result in Theorem~\ref{thm:A}(i).

We note that special cases of some of our results have previously been shown by Turner \cite{turner}, Hindry \cite{hi}, and Werner \cite{we} using rigid analytic methods. The role of the Berkovich skeleton is certainly hidden behind the scenes in these approaches and, by letting it enter the stage, a more complete picture arises. 

\subsection{Overview of the paper}

In \S\ref{sec:non-arch} we review a number of constructions and results related to uniformization of abelian varieties over non-archimedean fields.
In \S\ref{sec:taut_metric} we connect the canonical metric on line bundles on abelian varieties with their non-archimedean uniformization. We also give our definition of normalization of tropicalized theta functions and establish some of their elementary properties.
In \S\ref{sec:pp} we elaborate on the case that the line bundle at hand defines a principal polarization, and prove Theorem~\ref{thm:B2}.
In \S\ref{sec:taut_Mumford} we discuss the tautological Mumford models associated with the tropicalizations of non-archimedean theta functions.
In \S\ref{sec:Neron} we give a proof of Theorem~\ref{thm:A}.
In \S\ref{sec:normalized} we discuss a canonical way of normalizing non-archimedean canonical local heights.
The purpose of \S\ref{sec:elliptic} is to make things explicit for elliptic curves.
In \S\ref{sec:local-global} we present our explicit local-to-global formula for the canonical global height.

\subsection*{Acknowledgments}
We thank Ana Botero, Jos\'e Burgos Gil, David Holmes and Steffen M\"uller for helpful discussions.
The second-named author was partially supported by NSF CAREER DMS-2044564 grant.

\subsection*{Notation and convention} 

By $\log$ we will always mean the natural logarithm.

Throughout $R$ will be a complete discrete valuation ring, with fraction field $F$. The absolute value on $F$ is normalized such that a uniformizer has valuation one, so the value group is identified with $\zz$.

For an abelian variety $B$ over $F$, we denote its dual abelian variety by $B^t$. Let $P$ denote the rigidified Poincar\'e line bundle over $B \times B^t$.
Assume $M$ is a rigidified {\em ample} line bundle on $B$ and $\varphi_M \colon B \to B^t$ is the polarization determined by $M$. We have a canonical isomorphism of rigidified line bundles
\begin{equation} 
\label{eqn:canonical_rigid}
m^* M  \otimes p_1^*M^{\otimes -1} \otimes p_2^* M^{\otimes -1} = (\id,\varphi_M)^*P 
\end{equation}
over $B \times B$, where $p_1, p_2 \colon B \times B \to B$ denote the two projections and where $m \colon B \times B \to B$ is the addition map.

If $\bb$ is an abelian scheme over $R$ with generic fiber an abelian variety $B$ over $F$, the functor ``restriction to the generic fiber'' from the category of rigidified line bundles on $\bb$ to the category of rigidified line bundles on $B$ is an equivalence of categories  (\cite[Th\'eor\`eme II.1.1(ii)]{mb}). On an abelian scheme, the categories of rigidified line bundles and cubical line bundles are equivalent (\cite[pp.\ 8--9]{fc}). When $M$ is a rigidified line bundle on $B$, we denote by $\|\cdot\|_M$ the model metric on $M$ coming from the unique rigidified (equivalently, cubical) line bundle on $\bb$ that extends $M$.

\renewcommand*{\thethm}{\arabic{section}.\arabic{thm}}

\section{Non-archimedean uniformization} \label{sec:non-arch}

Let $A$ be an abelian variety over the field~$F$, and assume it has split semi-abelian reduction over $R$. 
The purpose of this section is to review a number of constructions and results dealing with the non-archimedean uniformization of $A$. 

\subsection{Raynaud extensions}
Choosing a polarization of~$A$, the Raynaud extension construction (as explained in \cite[\S II.1--2]{fc}) canonically attaches to these data a semi-abelian formal scheme
\begin{equation} \label{eqn:semiab_formal}
 1 \to \tt_\for \to \ee_\for \to \bb_\for \to 0 
\end{equation}
over $R$. This uniquely algebraizes to give a semi-abelian scheme, independent of the chosen polarization,
\begin{equation} \label{eqn:semiab}
 1 \to \tt \to \ee \to \bb \to 0 
\end{equation}
over $R$, with $\tt$ a split torus, and $\bb$ an abelian scheme. We write the generic fiber of this semi-abelian scheme as
\begin{equation} \label{eqn:semiab_generic}
 1 \to T \to E \to B \to 0 \, . 
\end{equation}
Upon base changing the semi-abelian formal scheme \eqref{eqn:semiab_formal} to the ring of integers of the completion of an algebraic closure of $F$, and taking generic fibers, we arrive at an extension of Berkovich analytic groups
\begin{equation} \label{eqn:semiab_Berko}
 1 \to T^\an \to E^\an \stackrel{q}{\to}  B^\an \to 0 \, . 
\end{equation}
We set
\[ X = \Hom(T,\gg_\m) \, , \!\quad Y = \Hom(T^t,\gg_\m) \, , \!\quad X^*=\Hom(X,\zz) \, , \!\quad Y^* = \Hom(Y,\zz) \, , \]
where $T^t$ denotes the dual torus of $T$. 

Let $u \in X$ and denote by $\chi^u \colon T \to \gg_\m$ the corresponding character. From \eqref{eqn:semiab_Berko} we obtain a pushout diagram of Berkovich analytic groups
\begin{equation}  \label{eqn:pushout} \xymatrix{ 1 \ar[r] & T^\an  \ar[d]^{\chi^u} \ar[r] & E^\an \ar[d]^{e_u} \ar[r]^{q} & B^\an\ar@{=}[d]  \ar[r] & 0  \\ 
1 \ar[r] & \gg_\m^\an \ar[r] & E_u^\an \ar[r] & B^\an \ar[r] & 0  }   
\end{equation}
with $E_u$ the $\gg_\m$-torsor associated to a canonical flat rigidified line bundle, also denoted $E_u$, on $B$. This pushout construction gives rise to a homomorphism $\varOmega' \colon X \to B^{t,\an}$ sending $u \in X$ to the equivalence class of $E_u^\an$. Here $B^t$ is the dual abelian variety of $B$.  In the same spirit, we have a homomorphism $\varOmega \colon Y \to B^\an$.

The homomorphisms $ \varOmega' \colon X \to B^{t,\an}$ and $\varOmega \colon Y \to B^\an$ lift to maps $X \to E^{t,\an}$ and $Y \to E^\an$ fitting in \emph{Raynaud crosses}
 \begin{align} \label{eq:crosses}
 \xymatrix{ 
& & Y \ar[rd]^{\varOmega} \ar[d] & & \\
&T^\an \ar[r] & E^\an \ar[r]^q \ar[d]^p & B^\an & \\
& & A^\an & &}  
 \xymatrix{
& & X \ar[rd]^{\varOmega'} \ar[d] & & \\
&T^{t,\an} \ar[r] & E^{t,\an} \ar[r] \ar[d] & B^{t,\an} & \\
 & & A^{t,\an} & &} 
 \end{align}
in the Berkovich analytic category. 

Let $P$ denote the rigidified Poincar\'e line bundle over $B \times B^t$. For $u' \in Y$ and $u \in X$, we have a canonical identification of $1$-dimensional vector spaces
\[ P|_{\varOmega(u') \times \varOmega'(u)} \cong E_u|_{\varOmega(u')} \, . \]
This gives rise to a map
\[ t \colon Y \times X \to P^\an \, , \quad (u',u) \mapsto e_u(u') \]
inducing a trivialization of $(\varOmega,\varOmega')^*P^\an$ over $Y \times X$.
The resulting map
\[ b \colon Y \times X \to \zz \, , \quad (u',u) \mapsto -\log \| t(u',u) \|_P  \]
is bilinear and non-degenerate, hence realizes $Y$ as a cofinite subgroup of $X^*$.

\subsection{Tropicalization maps} \label{subsec:trop_maps}
Let  $\langle \cdot,\cdot \rangle \colon X \times X^* \to \zz$ denote the evaluation pairing. It naturally extends as a map $X \times X_\rr^* \to \rr$, also denoted by $\langle \cdot,\cdot \rangle$, which is $\rr$-linear in the second coordinate. The {\em tropicalization map} $\val \colon T^\an \to X_\rr^*$ is given by the rule
\begin{equation}
 \langle u, \val(z) \rangle = -\log|\chi^u(z)| \, , \quad u \in X \, , \, z \in T^\an \, . 
\end{equation}
The map $\val \colon T^\an \to X_\rr^*$ is surjective, and extends naturally to give a surjective homomorphism $ E^\an \to X_\rr^*$, also denoted $\val$, by setting
\begin{equation} \label{def_trop}
 \langle u, \val(z) \rangle = -\log \|e_u(z) \|_{E_u} \, , \quad u \in X \, , \, z \in E^\an \, . 
\end{equation}
Here, the rigidified line bundle $E_u$ and the map $e_u \colon E^\an \to E^\an_u$ are as in \eqref{eqn:pushout}. 
The map $\val$ extends the identity map on $Y$. 

Write $\varSigma = X_\rr^*/Y$. This is a real torus.
We denote by $\overline{\val} \colon A^\an \to \varSigma$  the unique map making the following  diagram commutative:
\begin{equation} \label{def_tau} \xymatrix{    0 \ar[r] & Y \ar[r] \ar@{=}[d] & E^\an \ar[r]^p \ar[d]^{\val} & A^\an   \ar[r] \ar[d]^{\overline{\val}} & 0 \\
0 \ar[r] & Y \ar[r] & X_\rr^* \ar[r] & \varSigma \ar[r] & 0  } 
\end{equation}

\subsection{Rigidified ample line bundles} \label{subsec:rigidified}

By \cite[Theorem~6.7]{bl} (see also \cite[Theorem~3.6]{frss}), to give a rigidified ample line bundle $L$ on $A$ is to give a triple $(M,\varPhi,c)$ with
\begin{itemize}
\item $M$ a rigidified ample line bundle on $B$, 
\item $\varPhi \colon Y \to X$ an injective group homomorphism, 
\item $c \colon Y \to M$ a map over $B$ compatible with the rigidification of $M$, 
\end{itemize}
satisfying the conditions
\begin{equation}\label{eq:rel1}
 \varOmega' \circ \varPhi = \varphi_M \circ \varOmega
\end{equation}
and
\begin{equation}\label{eq:rel2}
 \forall u_1', u_2' \in Y : 
c(u_1' + u_2') \otimes c(u_1')^{-1} \otimes c(u_2')^{-1} = t(u_1', \varphi_M(u_2')) \, . 
\end{equation}
As before, $\varphi_M \colon B \to B^t$ is the polarization determined by the ample line bundle~$M$. The equality in \eqref{eq:rel2} holds under the canonical identification of rigidified line bundles given in \eqref{eqn:canonical_rigid}. The proof of the equivalence is done by descent with respect to the subgroup $Y \subset E^\an$. In particular, whenever $(M,\varPhi,c)$ is a triple for $L$, we have an identification $p^*L^\an = q^*M^\an$ of rigidified line bundles on $E^\an$. 

\begin{remark} \label{rmk:triples for L}
When $(M,\varPhi,c)$ and $(M',\varPhi',c')$ are triples for $L$, there exist $u \in X$, an isomorphism of rigidified line bundles $M' = M \otimes E_u$ over $B$, and an equality of trivializations $c' = c \otimes \varepsilon_u$ over $Y$, where $\varepsilon_u \colon Y \to E_u$ is the restriction of the map $e_u \colon E^\an \to E_u^\an$ to the subgroup $Y \subset E^\an$. Moreover, we have $\varPhi= \varPhi'$.
\end{remark}

Fix a rigidified ample line bundle $L$ and a triple $(M,\varPhi,c)$ for $L$. We define
\begin{equation} \label{eq:c_trop1}
c_\trop \colon  Y \to \zz \, , \quad  u' \mapsto -\log \| c(u') \|_M \, .      
\end{equation}
The canonical isomorphism \eqref{eqn:canonical_rigid} is an isometry when both sides are equipped with the metrics induced from the model metrics $\|\cdot\|_M$ and $\|\cdot\|_P$. This translates \eqref{eq:rel2} into the equality
\begin{equation} \label{eqn:c_trop_identity}
 \forall u_1', u_2' \in Y : c_\trop(u_1'+u_2') - c_\trop(u_1') -c_\trop(u_2') = b(u_1',\varPhi(u_2')) \, . 
\end{equation}
The map $Y \times Y \to \zz$ given by $(u_1',u_2') \mapsto b(u_1',\varPhi(u_2')) $ is bilinear, symmetric and positive definite. We have $c_\trop(0)=0$ because $c$ gives a trivialization compatible with the rigidification of $M$. It follows that $c_\trop \colon Y \to \zz$ can be written uniquely as a quadratic form plus a linear form. More precisely, we have 
\begin{equation}\label{eq:c_trop2}
c_\trop(u') = \frac{1}{2}b(u',\varPhi(u')) + \frac{1}{2}l(u') \, , \quad u' \in Y \, , 
\end{equation}
for a unique $l \in Y^*$. 

\subsection{Non-archimedean theta functions and their tropicalizations} \label{sec:theta_trop}

We continue to work with a rigidified ample line bundle $L$ on the abelian variety $A$.
Fix a triple $(M,\varPhi,c)$ for $L$ as in \S\ref{subsec:rigidified}. From now on, we also fix a non-zero global section $s \in H^0(A,L)$.

Following \cite[\S5]{bl} and \cite[\S3.6]{frss}, the global section $s \in H^0(A,L)$ uniquely determines a non-zero global section $f \in H^0(E^\an, q^*M^\an)$ that descends to $s$ along the uniformization map $p \colon E^\an \to A^\an$. The section $f$ is called the \emph{non-archimedean theta function} associated to $s$ (and the given triple). 

Following \cite[Example~7.2]{gu} (see, also, \cite[\S4.2]{frss}), there is a canonical section $\sigma \colon X_\rr^* \to E^\an$ of the  tropicalization map $\val \colon E^\an \to X_\rr^*$. The map $\sigma$ extends the identity map on $Y$. 

\begin{definition}{\cite[\S4.3]{frss}}
Using the section $\sigma$ of $\val$, one has the \emph{tropicalization} of $f$, given as the map
\begin{equation} \label{def:trop}
 f_\trop \colon X_\rr^* \to \rr \, , \quad \nu \mapsto - \log \| f(\sigma(\nu))\|_{q^*M} \, . 
\end{equation}
\end{definition}
\begin{definition} We define the map
\begin{equation} \label{eq:cocycle}
\mathbf{z} \colon Y \times X_\rr^* \to \rr \, , \quad
\mathbf{z}_{u'}(\nu) = c_\trop(u') + \langle \varPhi(u'),\nu \rangle \, , \quad u' \in Y \, , \, \nu \in X_\rr^*  
\end{equation}
attached to the triple $(M,\varPhi,c)$. 
The map $\mathbf{z}$ is a $1$-cocycle for the $Y$-action by translations on $X_\rr^*$, in the sense that the rule
\[ \forall u',v' \in Y \quad \forall \nu \in X_\rr^* \, : \quad  \mathbf{z}_{u'+v'}(\nu) = \mathbf{z}_{u'}(v'+\nu) + \mathbf{z}_{v'}(\nu) \]
is satisfied. This cocycle condition follows from \eqref{eqn:c_trop_identity}.
\end{definition}

\begin{prop} \label{prop:f_trop_properties} The following properties hold for the tropicalization~$f_\trop$ of $f$.
\begin{itemize}
\item[(i)] The function $f_\trop$ is continuous, piecewise affine and concave.
\item[(ii)] For any domain of linearity $\Delta$ of $f_\trop$ there exist $u \in X$ and $\ell \in \zz$ such that $f_\trop(\nu) = \langle u,\nu \rangle + \ell$ for $\nu \in \Delta$.
\item[(iii)] For all $u' \in Y$ and for all $\nu \in X_\rr^*$ we have $f_\trop(\nu) = f_\trop(\nu + u') + \mathbf{z}_{u'}(\nu)$.
\end{itemize}
\end{prop}
These properties are essentially shown in the proof of \cite[Theorem~4.7]{frss}. We briefly give some details. When $H$ is a line bundle on $B$, we write $\hh$ for the invertible sheaf of linear functionals on $H$. In particular, a section $h$ of the line bundle $H$ gives rise to a section of the invertible sheaf $\hh^{\otimes -1}$, also denoted~$h$. 

The non-archimedean theta function $f$ has a canonical Fourier decomposition,
\[ f = \sum_{u \in X} a_u \otimes e_u \, , \quad a_u \in H^0(B,\mm \otimes \ee_u) \, , \]
where $\mm$ is the invertible sheaf on $B$ corresponding to $M$, where $\ee_u$ is the invertible sheaf on $B$ corresponding to $E_u$, and $e_u \in H^0(B,\ee_u^{\otimes -1})$ for $u \in X$. This leads to the formula
 \begin{equation} \label{eqn:explicit_f_trop}
  f_\trop(\nu) = \min_{u \in X} \left\{ -\log \|a_u(\xi)\|_{M \otimes E_u} + \langle u,\nu \rangle \right\} < \infty \, , \quad \nu \in X_\rr^* \, .
 \end{equation}
Here $\xi$ is the Gauss point of $B^\an$, i.e., the Shilov point of $B^\an$ given by the generic point of the special fiber of the abelian scheme $\bb$. From the formula \eqref{eqn:explicit_f_trop} for $f_\trop$  properties (i) and (ii) readily follow.  Property (iii) follows from the transformation behavior of non-archimedean theta functions with respect to the $Y$-action as found in \cite[\S5]{bl} and revisited in \cite[Proposition~3.13]{frss}.

\begin{remark} \label{rmk:triples and f_trop}
The dependence of $f_\trop$ on the choice of triple is easy to describe. Let $u \in X$ be as in Remark~\ref{rmk:triples for L}. If the rigidified ample line bundle $M$ is changed into $M \otimes E_u$, and the trivialization $c$ into $c \otimes \varepsilon_u$, then the non-archimedean theta function $f \in H^0(E^\an, q^*M^\an)$ changes into $f \otimes e_u^{\otimes -1}$. The effect of this is that $f_\trop \colon X_\rr^* \to \rr$ changes by adding the global linear map on $X_\rr^*$ given by $-u \in X$. 
\end{remark}

\section{Normalized tropicalized theta functions} \label{sec:taut_metric}

In this section we connect the canonical metric on line bundles on abelian varieties with their non-archimedean uniformization. We also define the normalization of tropicalized theta functions (Definition~\ref{def:normaltheta}) and mention a few elementary properties.

We continue to work with a rigidified ample line bundle $L$ on $A$, as well as a triple $(M,\varPhi,c)$ for $L$. We will continue to use the non-archimedean uniformization of $A^\an$, as in \eqref{eq:crosses}.

\subsection{Canonical metric and uniformization} \label{subsec:canonical}

We denote by $\|\cdot\|_L$ the canonical metric on $L^\an$, see \cite[Example~3.7]{gu}. 
For instance, when $L$ is symmetric, and $n \in \zz_{\ge 2}$, we have a unique isomorphism of rigidified line bundles $[n]^*L \isom L^{\otimes n^2}$, and the canonical metric on $L^\an$ is the unique metric on $L^\an$ for which the induced isomorphism $[n]^*L^\an \isom L^{\an, \otimes n^2}$ becomes an isometry. 

\begin{example} Assume $A$ extends as an abelian scheme $\aa$ over $R$. Then $L$ extends uniquely as a rigidified ample line bundle $\ll$ on $\aa$. The canonical metric on $L^\an$ coincides with the model metric determined by $\ll$.  The canonical metric is, in general, not a model metric, but it is always an admissible metric \cites{zhsmall,gu_trop,cl_overview}.
\end{example}

We recall the map $\val \colon E^\an \to X_\rr^*$ from \eqref{def_trop}, and the map $c_\trop \colon Y \to \zz$ from \eqref{eq:c_trop1}. As seen in \eqref{eq:c_trop2}, we have the relation
\begin{equation} \label{eq:c_trop2_bis}
 c_\trop(u') = \frac{1}{2}b(u',\varPhi(u')) + \frac{1}{2}l(u')
\end{equation}
for a unique $l \in Y^*$. We denote by $c_{\trop,\rr} \colon X_\rr^* \to \rr$ the unique extension of $c_\trop$ as a quadratic map on the real vector space $X_\rr^*$. Note that (the homogeneous quadratic part of) $c_{\trop,\rr}$ is positive definite. It is readily verified from \eqref{eq:c_trop2_bis} that 
\begin{equation} \label{eq:cocycle_c_trop}
 c_{\trop,\rr}(\nu + u') = c_{\trop,\rr}(\nu) + \mathbf{z}_{u'}(\nu) \, , \quad \nu \in X_\rr^* \, , \quad u' \in Y \, .
\end{equation}

The result below follows from the discussion in Example~8.15 in \cite{gub_kun} (although not stated there explicitly). For the analogous result in the rigid analytic context (i.e., dealing with only the classical points of $A^\an$), we refer to~\cite[Th\'eor\`eme~C]{hi} and~\cite[Corollary 3.6]{we} in the situation where $A$ has multiplicative reduction, as well as~\cite[Th\'eor\`eme~D]{hi} in the general case. 

Let $s \in H^0(A,L)$ be a non-zero global section, and let $f \in H^0(E^\an,q^*M^\an)$ be the non-archimedean theta function determined by $s$ and the triple $(M,\varPhi,c)$, as in \S\ref{sec:theta_trop}.
Let $D$ be the divisor of the global section $s$ of $L$.
\begin{prop}{\cite[Example~8.15]{gub_kun}} \label{prop:ex_8.15} 
For all $x \in A^\an$ and $z \in E^\an$ with $p(z)=x$ and with $x \notin |D^\an|$ the equality
\[ -\log \|s(x)\|_L = c_{\trop,\rr}( \val (z)) - \log \| f(z) \|_{q^* M}  \]
holds.
\end{prop}

\subsection{Normalization of tropical theta functions}

We now come to a key definition in this paper. 
\begin{definition} \label{def:normaltheta}
We define the \emph{normalization} of the tropicalized theta function $f_\trop$  to be the function
\begin{equation} \| f_\trop\| = f_\trop + c_{\trop,\rr}
\end{equation}
on $X_\rr^*$.  
\end{definition}

\begin{remark}
The map $\|f_\trop\| \colon X_\rr^* \to \rr$ is independent of the choice of triple for $L$. To see this, let $u \in X$ be as in Remark~\ref{rmk:triples for L}. The map $c_{\trop,\rr}$ changes by adding the global linear map given by $u \in X$. As explained in Remark~\ref{rmk:triples and f_trop}, the function $f_\trop$ changes by adding the global linear map given by $-u$. We see that the sum $f_\trop + c_{\trop,\rr}$ is left unchanged. 
\end{remark}

\begin{prop} \label{prop:properties_norm_f_trop} The following properties hold for the normalized tropicalized theta function $\| f_\trop\|$.
\begin{itemize}
\item[(i)] The function $\|f_\trop\|$ is $Y$-invariant.
\item[(ii)] The function $\|f_\trop\|$ is piecewise quadratic, in fact, can be written as a minimum of positive definite quadratic functions. 
\item[(iii)] Let $f_1, f_2$ be non-archimedean theta functions for rigidified ample line bundles $L_1, L_2$ associated to triples $(M_1,\varPhi_1,c_1)$ and $(M_2,\varPhi_2,c_2)$. Then we have
\[ \| (f_1 \otimes f_2)_\trop \| = \|f_{1,\trop}\| + \|f_{2,\trop}\| \]
as functions on $X_\rr^*$.
\item[(iv)] Let $\nu \in X^*$, $N \in \zz_{\ne 0}$, and suppose that $N \nu \in Y$. Then $\| f_\trop\|(\nu) \in \frac{1}{2N} \zz$.
\end{itemize}
\end{prop}
\begin{proof} Property (i) follows by combining \eqref{eq:cocycle_c_trop} and Proposition~\ref{prop:f_trop_properties}(iii).
Property (ii) follows from equation~\eqref{eqn:explicit_f_trop} and the observations about $c_{\trop,\rr}$ made at the beginning of this section. Property (iii) is obvious. As to property (iv), it is clear from Proposition~\ref{prop:f_trop_properties}(ii) that $f_\trop(\nu) \in \zz$ for $\nu \in X^*$. From \eqref{eq:c_trop2_bis} it follows that $c_{\trop,\rr}(\nu) \in \frac{1}{2N}\zz$ if $\nu \in X^*$ and $0 \ne N\nu \in Y$.
\end{proof}

Consider the real torus $\varSigma = X_\rr^*/Y$ as in \S\ref{subsec:trop_maps}. The section $\sigma  \colon X_\rr^* \to E^\an$ induces a section $\iota \colon \varSigma \to A^\an$ of the map $\overline{\val} \colon A^\an \to \varSigma$. In the following we will usually think of the real torus $\varSigma$ as a subset of $A^\an$ via the embedding $\iota \colon \varSigma \to A^\an$; it is then called the \emph{canonical skeleton} of $A^\an$.

By Proposition~\ref{prop:properties_norm_f_trop}(i), we may consider $\|f_\trop\|$ as a continuous function on the canonical skeleton $\varSigma = X_\rr^*/Y$. 
The following result follows immediately from Proposition~\ref{prop:ex_8.15} and \eqref{def:trop}.
\begin{prop} \label{prop:f_trop_log_s} The identity of continuous functions
\[ \|f_\trop\| = -\log \|s\|_L \]
holds on the canonical skeleton $\varSigma=X_\rr^*/Y$.
\end{prop}

\section{The case of a principal polarization} \label{sec:pp}

We continue to work with a rigidified ample line bundle $L$ on $A$, as well as a triple $(M,\varPhi,c)$ for $L$. 
We assume in this section that $L$ defines a \emph{principal polarization}. Then, the line bundle $M$ defines a principal polarization on the abelian variety $B$, and the map $\varPhi \colon Y \to X$ is an isomorphism. It is customary to identify $Y$ and $X$ using this isomorphism.

We denote by $[\cdot,\cdot]$ the inner product on $X_\rr^*$ that is induced from the symmetric bilinear form $b$ on $X \times X$, and consider the continuous map $\varPsi \colon X_\rr^* \to \rr$ given by
\begin{equation} \label{eq:tropical_Riemann}
\varPsi(\nu) = \min_{u' \in Y} \left\{ \frac{1}{2}[u',u'] + [u',\nu] \right\} \, , \quad \nu \in X_\rr^* \, .
\end{equation}
This is a concave piecewise integral affine function with the property that for all $u' \in Y$ and for all $\nu \in X_\rr^*$ we have 
\begin{equation} \label{eqn:functional_eqn} \varPsi(\nu ) = \varPsi(\nu + u' ) + \mathbf{z}_{u'}(\nu) \, , 
\end{equation}
where  $\mathbf{z} \colon Y \times X_\rr^* \to \rr$ is the $1$-cocycle for the $Y$-action on $X_\rr^*$ given by 
\begin{equation} \label{eqn:trop_cocycle} \mathbf{z}_{u'}(\nu) = \frac{1}{2}[u',u'] + [u',\nu]  \, , \quad u' \in Y \, , \quad \nu \in X_\rr^* \, . 
\end{equation}
Based on the analogies with the classical complex analytic Riemann theta function one calls $\varPsi$ the \emph{tropical Riemann theta function} associated to the data $(X,Y,\varPhi,b)$. 
The maximal domains of linearity of $\varPsi$ are the Voronoi regions of the lattice $Y$ with respect to the inner product $[\cdot,\cdot]$.

Let $s \in H^0(A,L)$ be a non-zero global section of $L$ and let $f \in H^0(E^\an,q^*M^\an)$ be the non-archimedean theta function associated to~$s$ and the given triple $(M,\varPhi,c)$. 
\begin{prop}{\cite[Theorem~4.9]{frss}} \label{prop:tropical_Riem} Let $f_\trop$ be the tropicalization of $f$ as defined in \eqref{def:trop}. Then there exist a unique $k \in X_\qq^*$ and a unique $r' \in \qq$ such that $f_\trop = \varPsi \circ \t_k + r'$. Here $\t_k \colon X_\rr^* \to X_\rr^*$ denotes translation by $k$.
\end{prop}
We call $k \in X_\qq^*$ as in Proposition~\ref{prop:tropical_Riem} the \emph{tropical theta characteristic} of $f$. From \eqref{eqn:functional_eqn} and \eqref{eqn:trop_cocycle} combined with Proposition~\ref{prop:tropical_Riem} we find that
\begin{equation} \label{eqn:c_trop}
c_\trop(u') = \frac{1}{2}[u',u'] + [u',k] \, , \quad u' \in Y \, .  
\end{equation} 
By comparison with \eqref{eq:c_trop2} we see that $2k \in X^*$. 
\begin{remark}
By Remark~\ref{rmk:triples for L}, a change of triple for $L$ leads to a change of $c_\trop$ by the linear map $b(\cdot,u)$ for some $u \in X$.  It follows from \eqref{eqn:c_trop} that the tropical theta characteristic $k$ changes into $k+u'$ for some $u' \in Y$. In particular, the class $\kappa \in X_\qq^*/Y$ of $k$ mod $Y$ is independent of the choice of triple for $L$. We call the element $\kappa \in \varSigma$ the \emph{tropical theta characteristic} of $L$. Note that $2\kappa \in X^*/Y$. Moreover, $\kappa$ is a $2$-torsion point of $\varSigma$ if $L$ is \emph{symmetric}.
\end{remark}

We define the \emph{normalized tropical Riemann theta function} as the continuous piecewise quadratic map $\|\varPsi\| \colon X_\rr^* \to \rr$ given by
\begin{equation} \label{eq:normalized_tropical_Riemann}
\|\varPsi\|(\nu) = \varPsi(\nu) + \frac{1}{2} [\nu,\nu] \, , \quad \nu \in X_\rr^* \, . 
\end{equation}
We may simply write
\begin{equation} \label{eqn:norm_theta_simple}
\|\varPsi\|(\nu) = \frac{1}{2} \min_{u' \in Y} \, [\nu + u',\nu+u'] \, ,
\end{equation}
from which it is immediate that $\|\varPsi\|$ descends along the lattice $Y$ to give a map $\|\varPsi\| \colon \varSigma \to \rr$.

The following result establishes Theorem~\ref{thm:B2} from the introduction.
\begin{thm} \label{thm:B_bis} Assume that $L$ defines a principal polarization and let $f$ be a non-archimedean theta function for~$L$. Let $\kappa \in \varSigma$ be the tropical theta characteristic of $L$. There is a unique $r \in \qq$ such that the equality 
\[\|f_\trop\| = \|\varPsi\| \circ \t_\kappa + r \] 
holds on $\varSigma$. Here $\t_\kappa \colon \varSigma \to \varSigma$ denotes translation by $\kappa$.
\end{thm}
\begin{proof} Let $k \in X_\rr^*$ and $r' \in \qq$ be the invariants of $f$ as in Proposition~\ref{prop:tropical_Riem}. We compute, using \eqref{eqn:c_trop},
\begin{equation} \begin{split} \| f_\trop \|(\nu) & = f_\trop(\nu) + c_{\trop,\rr}(\nu) \\
  & = \varPsi(\nu +k) + \frac{1}{2}[\nu,\nu] + [\nu,k] + r' \\
  & = \|\varPsi\|(\nu + k ) - \frac{1}{2}[k,k] + r'
\end{split} 
\end{equation}
for $\nu \in X_\rr^*$. The required equality follows, with $r = - \frac{1}{2}[k,k] + r'$. 
\end{proof}

\section{Tautological models} \label{sec:taut_Mumford}

We continue to work with a rigidified ample line bundle $L$ on $A$, as well as a triple $(M,\varPhi,c)$ for $L$. We will continue to use the non-archimedean uniformization of $A^\an$, as in \eqref{eq:crosses}. The aim of this section is to explain that the tropicalizations of non-archimedean theta functions for $L$, as discussed in \S\ref{sec:theta_trop}, naturally give rise to certain Mumford models of the abelian variety $A$ and the line bundle $L$ over the ring of integers $R$.

The original source of Mumford's construction of models of the pair $(A,L)$ over~$R$ is \cite{mu}. The construction is explained in great detail in \cite[Chapter~III]{fc} and in \cite[Chapter~3]{ku}. These sources mostly work in the formal algebraic category. In \cite[Chapter~6]{bl} a discussion of Mumford's construction is given in terms of rigid analytic geometry, and in \cite[Chapter~4]{gu} in terms of Berkovich analytic geometry. 

\subsection{Mumford's construction}

We briefly describe the input and output of Mumford's construction, without saying too much about the construction itself.

Recall from \eqref{eq:cocycle} the $1$-cocycle
\begin{equation} 
\mathbf{z} \colon Y \times X_\rr^* \to \rr \, , \quad
\mathbf{z}_{u'}(\nu) = c_\trop(u') + \langle \varPhi(u'),\nu \rangle \, , \quad u' \in Y \, , \, \nu \in X_\rr^*  
\end{equation}
attached to the triple $(M,\varPhi,c)$.  Consider a pair $(\cc,\phi)$ consisting of:  
\begin{itemize}
\item[(a)] a $Y$-periodic rational polytopal decomposition  $\cc$ of the real vector space $X_\rr^*$, and 
\item[(b)] a convex continuous function $\phi \colon X_\rr^* \to \rr$, 
\end{itemize}
satisfying the following two conditions:
\begin{itemize}
\item[(i)] $\phi$ is piecewise integral affine with respect to $\cc$: for each $\Delta \in \cc$ there exist $u \in X$ and $\ell \in \zz$ with $\phi(\nu) = \langle u,\nu \rangle + \ell$ for $\nu \in \Delta$; 
\item[(ii)] $\phi$ has the given $1$-cocycle $\mathbf{z}$ as tropical automorphy factor: for all $u' \in Y$ and for all $\nu \in X_\rr^*$ we have $\phi(\nu + u') = \phi(\nu) + \mathbf{z}_{u'}(\nu)$.
\end{itemize}
We call $\phi$ an \emph{admissible polarization function} for $L$ with respect to  $\cc$. 
Mumford's construction attaches to the pair $(\cc,\phi)$ a pair $(\pp,\ll)$ consisting of:
\begin{itemize}
\item a projective flat model $\pp$ of the abelian variety $A$ over $R$, and
\item a line bundle $\ll$ on $\pp$ extending the line bundle $L$ on $A$. 
\end{itemize}
Let $G$ denote the connected component of identity of the N\'eron model of $A$ over $R$.
Mumford's construction starts from a suitable \emph{relatively complete model} for the degeneration data determined by the semi-abelian group scheme $G$ over $R$ and the triple $(M,\varPhi,c)$. 
Consider the short exact sequence 
\begin{equation} 
 1 \to \tt \to \ee \to \bb \to 0 
\end{equation}
of semi-abelian schemes over $R$ given in \eqref{eqn:semiab}.

Mumford's relatively complete model is obtained as a contraction product of the semi-abelian scheme $\ee$ with a torus embedding of the toric part $\tt$ of $\ee$, based on the rational cone decomposition of $(X_\rr^* \times \rr_{> 0}) \cup \{0\}$ formed by the cones over the polytopes in $\cc$, seen as subsets of $X_\rr^* \times \{1\}$.

We list a few properties of the Mumford model $(\pp,\ll)$ associated to $(\cc,\phi)$, all of which can be found in \cite[Chapter~3]{ku}. We use a subscript $(\cdot)_0$ to denote special fibers:
\begin{itemize}
\item the scheme $\pp$ is reduced and irreducible;
\item the scheme $\pp$ naturally contains the semi-abelian group scheme $G$ as an open subscheme;
\item the action of $G$ on itself by translations extends into an action of $G$ on $\pp$;
\item the orbits of the $G_0$-action on $\pp_0$ are semi-abelian varieties with abelian part $\bb_0$, and define a stratification of $\pp_0$;
\item the strata of $\pp_0$ are in bijective order reversing correspondence with the open faces of the polytopal decomposition $\cc/Y$ of $\varSigma=X_\rr^*/Y$;
\item the line bundle $\ll$ is relatively ample if the admissible polarization function $\phi$ is {\em strongly polyhedral} with respect to $\cc$, in the sense that the maximal-dimensional polytopes in $\cc$ are the maximal subsets of $X_\rr^*$ where $\phi$ is affine.
\end{itemize}
Let $(\pp,\ll)$ be the Mumford model associated to a pair $(\cc,\phi)$ as above.
Let $\|\cdot\|_\ll$ denote the model metric on $L$ determined by the line bundle $\ll$ on the projective flat integral model $\pp$ of $A$. 
We have the following fundamental result due to Gubler. 

\begin{prop}{\cite[Proposition~4.11]{gu}} \label{lem:gubler} The admissible polarization  function $\phi \colon X_\rr^* \to \rr$ and the model metric $\|\cdot\|_\ll$ satisfy the equality
\[ \phi \circ \val = - \log \circ \left( p^* \|\cdot\|_\ll / q^* \|\cdot \|_M \right) \, , \]
where the quotient of the metrics on $p^*L^\an = q^*M^\an$ is evaluated at any non-zero local section.
\end{prop}

\subsection{Tautological models}

Let $s \in H^0(A,L)$ be a non-zero global section, fixed from now on, and let $f$ be the non-archimedean theta function for $L$ determined by the section $s$ and the given triple $(M,\varPhi,c)$. 

Put $\phi = -f_\trop$. It follows from Proposition~\ref{prop:f_trop_properties} that $\phi \colon X_\rr^* \to \rr$ is convex continuous and satisfies the properties (i) and (ii) above.  Let $\cc$ be the set of maximal domains of linearity of $\phi$. Then, by construction, the function $\phi$ is a strongly polyhedral admissible polarization function for $L$ with respect to $\cc$. 

\begin{example} \label{exa:principal_pol_phi} Assume for the moment that $L$ defines a principal polarization. We resume the notations from \S\ref{sec:pp}. Let $k \in X_\qq^*$ be the theta characteristic of $f$. We obtain from Proposition~\ref{prop:tropical_Riem} that $\phi = - \varPsi \circ \t_k$, up to an additive constant, where $\varPsi \colon X_\rr^* \to \rr$ is the tropical Riemann theta function determined by the inner product $[\cdot,\cdot]$ on $X_\rr^*$. In particular the polytopal decomposition $\cc$ equals the shift by $k$ of the Voronoi tiling of the Euclidean space $X_\rr^*$ induced by the lattice $Y \subset X_\rr^*$.
\end{example}

\begin{definition} \label{def:taut_model} The \emph{tautological model} of the pair $(A,L)$ with respect to the given global section~$s$ is the model $(\pp,\ll)$ of the pair $(A,L)$ given by applying Mumford's construction to the pair $(\cc, \phi)$ where $\phi = -f_\trop$ and $\cc$ is the set of maximal domains of linearity of $\phi$ as above. Note that the line bundle $\ll$ on $\pp$ is relatively ample.
\end{definition}

\begin{remark} Up to isomorphisms, the tautological model of the pair $(A,L)$ is independent of the choice of triple for $L$. This follows from Remark~\ref{rmk:triples and f_trop}, which states that the effect of a change of triple is that $f_\trop$ changes by adding a global linear map. In particular the polytopal decomposition $\cc$ is unchanged. From the way the model $(\pp,\ll)$ is constructed in \cite[Chapter~3]{ku}, using the technology of torus embeddings over $R$, we see that the model $\pp$ is unchanged, and the line bundle $\ll$ changes only by a \emph{trivial} line bundle. 
\end{remark}

\begin{definition} \label{def:semi-stable}
Let $(\pp,\ll)$ be the tautological model of the pair $(A,L)$ over $R$ with respect to the given global section~$s$. We call the model $(\pp,\ll)$ \emph{semistable} if the special fiber $\pp_0$ is \emph{reduced}. The discussion in \cite[\S4]{ku} gives the following statements:
\begin{itemize}
\item[(i)] when $(\pp,\ll)$ is semistable over $R$, the formation of the tautological model of $(A,L)$ over finite extensions of $R$ is compatible with base change;
\item[(ii)] there exists a finite extension $R' \supset R$ such that the tautological model of $(A,L)$ over $R'$ is semistable.
\end{itemize}
\end{definition}

\subsection{The tautological model metric}

Let $(\pp,\ll)$ be the tautological model of the pair $(A,L)$ with respect to the section~$s$. 
The purpose of this section is to present some properties of the model metric $\|\cdot\|_\ll$  on $L^\an$ determined by the model $\ll$ of the line bundle $L$. 

Let $\|\cdot\|_L$ denote the canonical metric on $L^\an$, see \S\ref{subsec:canonical}.
The first result describes the quotient $\|\cdot\|_\ll/\|\cdot\|_L$ as a function on $A^\an$. 
\begin{prop}  \label{thm:quotient} Let $f$ be a non-archimedean theta function for the global section~$s$, and let $\|f_\trop\| \colon \varSigma \to \rr$ be the normalized tropicalization of $f$. Then the equality
\begin{equation} \label{eq:metric_vs_ftrop}
\log \circ \left( \|\cdot\|_\ll/\|\cdot\|_L \right) = \| f_\trop \| \circ \overline{\val} 
\end{equation}
holds on $A^\an$, where the quotient of the metrics is evaluated at any non-zero local section.
\end{prop}
We observe that Proposition~\ref{thm:quotient} shows, rather explicitly, that the canonical metric on $L^\an$ is \emph{continuous} and \emph{bounded} in the sense of \cite[\S1.1]{zhsmall}: the ratio with the tautological model metric on $L^\an$ is controlled by a continuous function on the canonical skeleton $\varSigma \subset A^\an$, which is compact.

\begin{proof}[Proof of Proposition~\ref{thm:quotient}] Let $D$ be the divisor of the global section $s$ of $L$.
Let $x \in A^\an$ and $z \in E^\an$ with $p(z)=x$ and assume that $x \notin |D^\an|$. By Proposition~\ref{lem:gubler} we have
\[ f_\trop(\val(z)) = \log \| s(x) \|_\ll - \log \|f(z)\|_{q^*M} \, . \]
By Proposition~\ref{prop:ex_8.15} we have
\[ - \log \| f(z) \|_{q^* M} = -\log \|s(x)\|_L - c_{\trop,\rr}( \val (z))  \, . \]
Combining we find 
\[ \|f_\trop\|(\val(z)) = \log \| s(x) \|_\ll - \log \|s(x)\|_L \, . \]
This proves the equality \eqref{eq:metric_vs_ftrop} on the dense open subset $A^\an \setminus |D^\an|$. The result follows by continuity of the functions at each side of the identity.
\end{proof}

Combining Proposition~\ref{thm:quotient} with Proposition~\ref{prop:f_trop_log_s} we obtain the following interesting vanishing property of the model metric.
\begin{cor}  \label{lem:identical_vanish} The function $\log \|s\|_\ll$ vanishes identically on the canonical skeleton $\varSigma \subset A^\an$.
\end{cor}
Combining Proposition~\ref{thm:quotient} with Theorem~\ref{thm:B_bis} we obtain the following more explicit result in the principally polarized case.
\begin{cor} \label{cor:quotient_pp} Assume that $L$ defines a principal polarization. Let $\kappa \in \varSigma$ be the tropical theta characteristic of~$L$. Then there exists a unique $r \in \qq$ such that the equality
\[ \log \circ \left( \|\cdot\|_\ll/\|\cdot\|_L \right) = \| \varPsi \| \circ \t_\kappa \circ  \overline{\val} + r \]
holds on $A^\an$. Here  the quotient of the metrics is evaluated at any non-zero local section, and $\|\varPsi\|$ is the normalized tropical Riemann theta function as in~\eqref{eq:normalized_tropical_Riemann}.
\end{cor}

\subsection{Tautological divisors}

Let $(\pp,\ll)$ be the tautological model of the pair $(A,L)$ with respect to the section~$s \in H^0(A,L)$. We may view $s$ as a rational section of the line bundle $\ll$ and consider the Cartier divisor $\dd = \divisor_\ll(s)$ on the scheme $\pp$. 

\begin{prop}  \label{thm:rel_eff}   The divisor $\dd$ is an effective relative Cartier divisor.
\end{prop}
\begin{proof}
Let $\Sk(\pp)$ be the Berkovich skeleton of the Mumford model $\pp$ inside $A^\an$ (see, e.g., \cite[\S5.3]{gu}). By \cite[Example~7.2]{gu}, the 
natural inclusion $\iota \colon \varSigma \to A^\an$ is a homeomorphism onto  $\Sk(\pp)$. Let $V$ be an irreducible component of the special fiber $\pp_0$ of $\pp$ and let $x_V \in \Sk(\pp)$ be the Shilov point corresponding to $V$. We see that $x_V \in \iota(\varSigma)$ and, by Corollary~\ref{lem:identical_vanish}, we have
\[ \ord_{V,\ll}(s) = -\log \|s(x_V)\|_\ll = 0 \, . \]
This shows that the Cartier divisor $\dd$ has multiplicity zero along $V$. In other words, no local equation of the Cartier divisor $\dd$ becomes a zero divisor upon reduction to $\pp_0$. This yields that the Cartier divisor $\dd$ is flat over $R$ by, for example, the discussion in \cite[Example~III.9.8.5]{hag}.
\end{proof}
Write $D=\divisor_L(s)$ for the divisor of $s$ on $A$.
\begin{definition}
 We call the pair $(\pp,\dd)$ with $\dd$ as above the \emph{tautological model} of the pair $(A,D)$. Note that the relative Cartier divisor $\dd$ is relatively ample, since the line bundle $\ll$ is relatively ample.
\end{definition}

\subsection{Connection with the models of Alexeev--Nakamura}

In the paper \cite{an} by Alexeev and Nakamura one finds a detailed construction and discussion of projective flat models $(\pp,\varTheta)$ of pairs $(A,D)$ as above in the case that $L$ defines a principal polarization. In the Alexeev--Nakamura model $(\pp,\varTheta)$ of the pair $(A,D)$, the scheme $\pp$ is a semistable Mumford model of $A$ over $R$, and $\varTheta$ is a relative effective Cartier divisor on $\pp$ extending the divisor $D$. 

A closer look at the construction in loc.\ cit.\ yields that the Alexeev--Nakamura model $(\pp,\varTheta)$ of $(A,D)$ is obtained by applying Mumford's construction starting from the pair $(\cc,\phi)$, where $\cc$ is the Voronoi tiling of $X_\rr^*$ shifted by a theta characteristic $k$ of $L$, and where $\phi$ equals the admissible polarization function $-\varPsi \circ \t_k$. By Example~\ref{exa:principal_pol_phi} we conclude that the tautological model of $(A,D)$, if it is semistable, coincides with the Alexeev--Nakamura model of $(A,D)$. In the case of higher degree polarizations, one expects to reobtain the models mentioned in \cite[Remark~5.B]{an}.

\begin{thm} Assume the rigidified line bundle $L$ defines a principal polarization. Let $D$ be the unique effective divisor on $A$ determined by $L$.  Let $(\pp,\varTheta)$ be the Alexeev--Nakamura model of the pair $(A,D)$, and set $\ll_0= \oo_\pp(\varTheta)$. 
Let $\|\cdot\|_L$ denote the canonical metric on $L$. 
Let $\kappa \in \varSigma$ be the tropical theta characteristic of~$L$. There is a unique $r'' \in \qq$ such that the equality
\[ \log \circ \left( \|\cdot\|_{\ll_0}/\|\cdot\|_L \right) = \| \varPsi \| \circ \t_\kappa \circ  \overline{\val} + r'' \]
holds on $A^\an$. Here the quotient of the metrics is evaluated at any non-zero local section, and $\|\varPsi\|$ is the normalized tropical Riemann theta function as in~\eqref{eq:normalized_tropical_Riemann}.
\end{thm}
\begin{proof} Let $s \in H^0(A,L)$ be any non-zero global section of $L$. Let $(\pp,\ll)$ be the tautological model of $(A,L)$ determined by the section $s$. From Proposition~\ref{thm:rel_eff} we obtain that $\divisor_\ll(s) = \varTheta $. We deduce that there exists a generically trivial line bundle $N$ over $\Spec R$ and an isomorphism $\ll \cong \ll_0 \otimes \pi^*N$ extending the identity map on $L$ over $\pp$. Here $\pi \colon \pp \to \Spec R$ denotes the structure map. We see that there is an $a \in \zz$ with $-\log \| \cdot \|_{\ll} = -\log \|\cdot\|_{\ll_0} + a$. We obtain the required equality upon invoking Corollary~\ref{cor:quotient_pp}.
\end{proof}

\section{Proof of Theorem~\ref{thm:A}} \label{sec:Neron}

We recall the setup and the statement for convenience. Let $A$ be an abelian variety over $F$ and let $L$ be a rigidified ample line bundle on $A$. Let $\|\cdot\|_L$ denote the canonical metric on~$L$. Let $s \in H^0(A,L)$ be a non-zero global section and write $D=\divisor_L(s)$. Let $f$ be a theta function associated to~$s$, let $\varSigma \subset A^\an$ be the canonical skeleton of $A^\an$, and let $\|f_\trop\| \colon \varSigma \to \rr$ be the normalized tropicalization of $f$ (Definition~\ref{def:normaltheta}). 

Let $\nn$ be the N\'eron model of $A$ over $R$ and let $\varPhi_\nn$ denote the group of connected components of $\nn$. Let $\sp \colon A(F) \to \varPhi_\nn$ denote the specialization map. As we have seen, the map $\varPhi_\nn \to X^*/Y$ sending $\sp(x)$ for $x \in A(F)$ to $\overline{\val}(x)$ is a group isomorphism. 
Therefore, we view the group $\varPhi_\nn$ canonically as a subgroup of the canonical skeleton $\varSigma$.
 
\begin{thm} \label{thm:A_bis} Let $N \in \zz_{>0}$ be any positive integer such that $N \cdot \varPhi_\nn =0$.
\begin{itemize} 
\item[(i)] The restriction of the function $\|f_\trop\| $ to $\varPhi_\nn$ takes values in $ \frac{1}{2N}\zz$.
\item[(ii)] Let $x \in A(F) \setminus |D|$ and write $x_0=\sp(x)$. The equality
\begin{equation} \label{Neron_bis}
 -\log \|s(x)\|_L = \mathbf{i}(x,D) + \|f_\trop\| (x_0) 
 \end{equation}
holds. Here  $\mathbf{i}(x,D)$ denotes the intersection multiplicity of the Zariski closure of $x$ in $\nn$ with the thickening of $D$ over $\nn$. 
\end{itemize}
\end{thm}
\begin{proof}  Property (i) follows from Proposition~\ref{prop:properties_norm_f_trop}(iv). To prove property (ii), let $(\pp,\ll)$ be the tautological model of the pair $(A,L)$ with respect to the global section $s$ (Definition~\ref{def:taut_model}), and write $\dd = \divisor_\ll(s)$. Then $\pp$ is projective over $R$ and $\dd$ is a relative effective Cartier divisor on $\pp$, by Proposition~\ref{thm:rel_eff}. Let $\overline{x}_\pp$ be the section of $\pp$ over $\Spec R$ extending the point $x \in A(F)$, by the valuative criterion of properness. Let $v$ denote the closed point of $\Spec R$ and let $(\overline{x}_\pp,\dd)$ denote the multiplicity of the Cartier divisor $\overline{x}_\pp^*\dd$ at $v$. 

We have $(\overline{x}_\pp,\dd) = -\log \|s(x)\|_\ll$, by definition of model metrics. From Proposition~\ref{thm:quotient} we obtain the identity
\begin{equation} \label{eqn_almost_thmD}
-\log \|s(x)\|_L = (\overline{x}_\pp,\dd) + \| f_\trop \|(x_0) \, . 
\end{equation}
Comparing with \eqref{Neron_bis}, we are therefore reduced to showing the equality of integers
\begin{equation} \label{eqn:to_show}
\mathbf{i}(x,D) = (\overline{x}_\pp,\dd) \, . 
\end{equation}
Consider the short exact sequence 
\begin{equation}
 1 \to \tt \to \ee \to \bb \to 0 
\end{equation}
of semi-abelian schemes over $R$ given in \eqref{eqn:semiab}.
We recall from \S\ref{sec:taut_Mumford} that Mumford's construction resulting in the tautological model $(\pp,\ll)$ is done using a torus embedding of the toric part $\tt$ of $\ee$, based on the rational cone decomposition of $(X_\rr^* \times \rr_{> 0}) \cup \{0\}$ formed by the cones over the maximal domains of linearity of $f_\trop$, seen as a rational polytopal decomposition of $X_\rr^* \times \{1\}$. 

As is explained in \cite[\S3.1]{ku}, passing to a smooth refinement of this cone decomposition and applying again Mumford's construction produces a regular Mumford model $(\widetilde{\pp},\widetilde{\ll})$ of $(A,L)$ together with a map of $R$-models $\mu \colon \widetilde{\pp} \to \pp$ and a canonical identification $\mu^*\ll = \widetilde{\ll}$.  Writing $\widetilde{\dd} = \divisor_{\widetilde{\ll}}(s)$ we have the equality of Cartier divisors $\widetilde{\dd} = \mu^*\dd$ on $\widetilde{\pp}$.

Let $\overline{D}$ denote the thickening of $D$ over $\nn$.
The argument used to prove Proposition~\ref{thm:rel_eff} also applies to the Mumford model $(\widetilde{\pp},\widetilde{\ll})$, showing that the divisor $\widetilde{\dd} $ is relative Cartier. As is explained in  \cite[\S4.4]{ku}, there exists an open immersion $j \colon \nn \to \widetilde{\pp}$ over $R$. By flatness of open immersions, we see that the divisor $ j^*\widetilde{\dd} = (\mu \circ j)^*\dd$ is relative Cartier, hence a thickening of $D$ over $\nn$. By uniqueness of thickenings we conclude that the equality $(\mu \circ j)^*\dd = \overline{D}$ of Cartier divisors holds on $\nn$.  

Let $\overline{x}_\nn$ denote the Zariski closure of $x$ in $\nn$. Note that $\overline{x}_\pp = \mu_*j_* \overline{x}_\nn$. We can now compute
\[ \begin{split} 
\mathbf{i}(x,D) 
& = (\overline{x}_\nn,\overline{D}) \\
 & = (\overline{x}_\nn, (\mu \circ j)^*\dd) \\
 &= \ord_v \overline{x}_\nn^* (\mu \circ j)^* \dd \\ 
 & = \ord_v \overline{x}_\pp^* \dd \\
 & =  (\overline{x}_\pp,\dd) \, , \end{split} \]
establishing \eqref{eqn:to_show}. 
\end{proof}

The following version of the formula in Theorem~\ref{thm:A_bis} is sometimes more convenient.
\begin{thm} \label{thm:alternative_thmA} Let $x \in A(F) \setminus |D|$ and write $x_0=\sp(x)$. The equality
\begin{equation}
 -\log \|s(x)\|_L = \mathbf{i}(x,D) -\log\|s (x_0)\|_L 
 \end{equation}
holds. 
\end{thm}
\begin{proof} Combine Theorem~\ref{thm:A_bis}(ii) with Proposition~\ref{prop:f_trop_log_s}.
\end{proof}

\begin{remark}
Recall from Proposition~\ref{prop:ex_8.15} that, for all $x \in A^\an$ and $z \in E^\an$ with $p(z)=x$ and $x \notin |D^\an|$, we have 
\[ -\log \|s(x)\|_L = c_{\trop,\rr}( \val (z)) - \log \| f(z) \|_{q^* M} \, . \]
Now for $x \in A^\an$ fixed and $z$ running though $E^\an$ with $p(z)=x$ we have that $c_{\trop,\rr}(\val(z))$ is bounded from below, and  $ - \log \| f(z) \|_{q^* M}$ bounded from above. This means that we could write
\[ -\log \|s(x)\|_L = \min_{z : p(z)=x} c_{\trop,\rr}( \val (z)) + \max_{z: p(z)=x} \{ - \log \| f(z) \|_{q^* M} \} \, . \]
Comparing with Theorem~\ref{thm:A_bis}, it is perhaps tempting to guess that there exists $t \in F^\times$ such that
the equalities
\[ \mathbf{i}(x,D) \stackrel{?}{=} \max_{z: p(z)=x} \{ - \log \| t f(z) \|_{q^* M} \} \, , \quad \| (tf)_\trop(x_0)\| \stackrel{?}{=}  \min_{z : p(z)=x} c_{\trop,\rr}( \val (z))  \]
hold true uniformly over all $x \in A(F) \setminus |D|$. Here we write $x_0 = \sp(x)$. However, such equalities can not be true in general: note that the function $\|f_\trop \|$ depends on the choice of global section $s$, whereas the function $\min_{z : p(z)=x} c_{\trop,\rr}( \val (z))$ is an invariant of the triple $(M,\varPhi,c)$ for $L$. The equalities above do hold true for suitable $t \in F^\times$ if $L$ defines a principal polarization.
\end{remark}

\section{Normalized canonical local heights} \label{sec:normalized}

We continue to work with the rigidified ample line bundle~$L$.
The presence of the canonical skeleton $\varSigma \subset A^\an$ suggests the following canonical way of renormalizing non-archimedean canonical local height functions. Let $\mu_H$ denote the Haar measure on the real torus $\varSigma$, with volume normalized to be equal to one. Denoting by $\iota \colon \varSigma \to A^\an$ the inclusion, we have a pushforward measure $\iota_*\mu_H$ on $A^\an$. In the following, we will denote this pushforward measure by $\mu_H$ as well.

Let $s \in H^0(A,L)$ be a non-zero global section of $L$ and set $D = \divisor_L(s)$. We recall that we have the  canonical local height function
\[ \lambda_D(x) = -\log \|s(x)\|_L \, , \quad x \in A^\an  \setminus |D^\an|  \, , \]
associated to the ample divisor~$D$. 
\begin{definition} \label{def:normalized} (A canonical normalization) 
 As follows from Proposition~\ref{prop:f_trop_log_s}, the Green's function $\lambda_D$ is $\mu_H$-integrable -- in fact, it restricts to a continuous function on $\varSigma$.  Thus we may renormalize the Green's function $\lambda_D$ by defining
\begin{equation} \label{eqn:def_our_normalized}
 \lambda'_D(x) = -\log \|s(x)\|_L +  \int_{A^\an} \log \|s\|_L \, \d \, \mu_H \, , \quad x \in A^\an \setminus |D^\an| \, .
\end{equation}
We observe that the Green's function $\lambda'_D$ is entirely intrinsic to the divisor~$D$. Indeed,  changing the global section $s$ that defines $D$ by a scalar multiple does not change the function $\lambda'_D$. Moreover, the function $\lambda'_D$ is independent of the choice of rigidification of~$L$. 
\end{definition}
It is easy to see the following identities:
\begin{itemize}
\item[] (additivity) let $\t_a \colon A \to A$ denote translation by $a$  for $a \in A(F)$. Then $\lambda'_{\t_a^*D} = \t_a^* \lambda'_D$; 
\vspace{2mm}
\item[] (translation) $\lambda'_{D_1} + \lambda'_{D_2} = \lambda'_{D_1+D_2}$
for ample divisors $D_1, D_2$ on $A$;
\vspace{2mm}
\item[] (isogenies) let $\varphi \colon A \to B$ be an isogeny of abelian varieties over $F$, and let $D$ be an ample divisor on $B$. Then $\lambda'_{\varphi^*D} = \varphi^* \lambda'_{D}$.
\end{itemize}

Let $\lambda'_D$ be the normalized canonical local height determined by $D$ as in Definition~\ref{def:normalized}.
The following proposition is immediate from Theorem~\ref{thm:alternative_thmA}. 
\begin{prop}  \label{prop:normalized_classical} Let $s$ be any global section of $L$ whose divisor is $D$.  Let $x \in A(F) \setminus |D|$ and write $x_0=\sp(x)$. Then the formula
\[ \lambda'_D(x) = \mathbf{i}(x,D) -\log\|s (x_0)\|_L  + \int_{A^\an} \log \|s\|_L \, \d \, \mu_H \]
holds.
\end{prop}
\begin{cor} \label{cor:normalized_good} Assume that $A$ has good reduction. Let $x \in A(F) \setminus |D|$. Then the formula
\[ \lambda'_D(x) = \mathbf{i}(x,D) \]
holds.
\end{cor}
By additivity we obtain a normalized canonical local height attached to any divisor $D$ on $A$.

\section{Elliptic curves} \label{sec:elliptic}

Let $A$ be an elliptic curve over $F$ and let $D$ be a divisor on $A$. In this section we study the  normalized canonical local height $\lambda'_D$ associated to $D$ as discussed in the previous section. By translation and additivity we may reduce to the case that $D=(O)$, where $O$ is the origin of $A$. 

\subsection{Comparison with Tate's normalization} As is well known, Tate already gave, in a 1968 letter to Serre, a canonical way of normalizing the canonical local height associated to $D=(O)$, see for instance \cite[Chapter 3, \S5]{langElliptic}, \cite[\S6.5]{lectures_mordell} or \cite[Chapter VI]{sil_advanced}. Let
\[ y^2 + a_1xy + a_3y = x^3 + a_2x^2 + a_4x + a_6 \, , \quad a_1, a_2, a_3, a_4, a_6 \in F \]
be a Weierstrass equation for $A$, let $\varDelta$ be its discriminant, and write $z=x/y$ which is a local coordinate near $O$. 
Tate's normalization of a canonical local height $\lambda_D$ on $A(F)$ associated to $D$ is determined by requiring that 
\begin{equation} \label{eqn:Tate_normalization}
 \lim_{p \to 0} \left\{  \lambda_D(p) + \log |z(p)|   \right\} = -\frac{1}{12}\log|\Delta|  
\end{equation}
should hold, where the limit is taken over $A(F) \setminus |D|$, in the non-archimedean topology. One quickly checks that Tate's normalization is independent of the choice of Weierstrass equation for $A$, and is hence intrinsic to the elliptic curve~$A$. 

In \cite[Theorem~VI.4.2]{sil_advanced} and \cite[\S6.5]{lectures_mordell} one finds explicit formulas for Tate's normalized canonical local height.  There are two cases to consider: the case where $A$ has good reduction, and the case where $A$ has split multiplicative reduction. 

A direct comparison between the explicit formulas in loc.\ cit.\ and the explicit formulas that we find below for $\lambda'_D$ reveals that -- perhaps unsurprisingly -- both normalizations end up giving the \emph{same}  canonical local height.
\begin{prop} Let $A$ be an elliptic curve over $F$, with origin $O$, and set $D=(O)$. Denote the canonical extension over $A^\an$ of Tate's normalized canonical local height by $\widehat{\lambda}_D$. Let $\lambda'_D$ be the normalized canonical local height associated to $D$ as in Definition~\ref{def:normalized}. The Green's functions $\lambda'_D$ and $\widehat{\lambda}_D$ coincide on $A^\an$.
\end{prop}
In other words, for the canonical extension $\widehat{\lambda}_D$ over $A^\an$ of Tate's normalized canonical local height on $A(F)$, we have the property that 
\[ \int_{A^\an} \widehat{\lambda}_D \, \d \, \mu_H = 0 \, . \]
A direct proof of this fact, without using explicit formulas, can be given by arguing along the lines of \cite[\S 5]{dj_local_Galois}.

\subsection{Good reduction} Assume that  $A$ has good reduction over $F$. Let $\nn$ be the N\'eron model of $A$ over $R$, which in this case is an abelian scheme over $R$. By Corollary~\ref{cor:normalized_good} we simply have 
\[ \lambda'_D(x) = \mathbf{i}(x,D) \, , \quad x \in A(F) \setminus |D| \, ,  \] 
with $\mathbf{i}(x,D)$ the intersection multiplicity between the Zariski closure of $x$ and the Zariski closure of $O$ on the arithmetic surface $\nn$.

\subsection{Split multiplicative reduction}

Assume next that $A$ has split multiplicative reduction over $F$, i.e., is a Tate elliptic curve. 
We may write $A^\an = \gg_\m^\an / q^\zz$ for some $q \in \gg_\m(F)$ satisfying $0<|q|<1$.  
We may identify $X$ with $\zz$ so that $X^*=\zz$ and $X_\rr^*=\rr$.  Let $L=\oo_A(D)$ and choose a trivialization of~$L$. As is well known, the function
\begin{equation} \label{theta_product}
 \theta(z) = (1-z) \prod_{n=1}^\infty (1-q^nz)(1-q^n z^{-1} ) \, , \quad z \in \gg_\m(F)
\end{equation}
gives a non-archimedean theta function for $L$.
Up to multiplication by an element of $R^\times$, the function $\theta(z)$ has Fourier expansion
\begin{equation} \label{fourier}
 \theta(z) \equiv \sum_{u \in \zz} (-1)^u q^{\frac{u^2-u}{2}} \chi^u \, , 
\end{equation}
where $\chi^u $ is the character corresponding to $u \in \zz$.

Note that $L$ defines a principal polarization. Therefore as usual we set $Y=X$ and let $\varPhi \colon Y \isom X$ be the identity map. In particular we have an identification $Y=\zz$.

Let $\oo$ denote the trivial line bundle of $\Spec F$. Following \cite[Example~3.16]{frss}, from the expansion \eqref{fourier} we obtain that $(\oo,\varPhi,c)$ is a triple for $L$ with the trivialization $c \colon \zz \to \oo$ given by 
\begin{equation} \label{formula:c}
 u' \mapsto (-1)^u q^{\frac{u^2-u}{2}} \, , \quad u = u' \, , \quad u' \in  \zz \, . 
\end{equation}
Note that for the purpose of computing $\lambda'_D$ the actual rigidification of $L$ is immaterial. Hence, in the following, we shall assume $\oo$ is equipped with the canonical rigidification.

\subsection{Tropicalizations}

Set $\ell= - \log|q| $. The tropicalization map $\val \colon \gg_\m^\an \to \rr$ sends $z$ to $-\log|z|$. We obtain an identification of the canonical skeleton $\varSigma$ with the real torus  $ \rr/\ell \zz$. We can now write down the tropicalizations of the map $c \colon \zz \to \oo$ and of the non-archimedean theta function $\theta$ as follows.

From \eqref{formula:c} we obtain
\begin{equation} \label{eqn:formula_c_trop} 
 c_\trop(u')  = \ell \cdot \frac{u^2-u}{2}  \, , \quad u = u' \, , \quad u' \in  \zz \, .
\end{equation}
This leads to
\begin{equation}  \begin{split} c_{\trop,\rr}(\nu) & = \ell \cdot \frac{(\nu/\ell)^2-\nu/\ell}{2} \\
& =  \frac{\ell}{2}B_2(\nu/\ell) - \frac{1}{12} \ell
\end{split} 
\end{equation}
for $\nu \in\rr$, where $B_2(T) = T^2-T+1/6$ is the second Bernoulli polynomial.

The bilinear map $b \colon X \times Y \to \zz$ is identified with the map $\zz \times \zz \to \zz$ given by $(u,u') \mapsto \ell uu'$. 
It follows that for the inverse inner product $[\cdot,\cdot]$ on $X_\rr^*=\rr$ we have $[\nu,\mu] = \nu\mu/\ell$ for $\mu, \nu \in \rr$. In these terms,
upon tropicalizing the Fourier expansion \eqref{fourier} we find
\begin{equation} \begin{split} \theta_\trop(\nu) & = \min_{u \in \zz} \left\{ \ell \cdot \frac{u^2-u}{2} + u \cdot \nu \right\} \\
 & = \min_{u' \in \zz} \left\{ \frac{1}{2} \left[\ell u',\ell u'\right] + \left[\ell u',\nu - \frac{\ell}{2} \right] \right\} \\
\end{split} 
\end{equation}
for $\nu \in \rr$. That is to say
\begin{equation} \label{eqn:theta_and_Psi}
 \theta_\trop = \varPsi \circ \t_{-\ell/2}  \, , 
\end{equation}
in accordance with Proposition~\ref{prop:tropical_Riem}. We see that the tropical theta characteristic of $L$ is the unique \emph{non-trivial} $2$-torsion point in $\rr/\ell\zz$.

Note that $\theta_\trop$ vanishes for $\nu \in [0,\ell]$. This leads to the formula
\begin{equation} \begin{split} \label{eqn:theta_trop_abs}
\|\theta_\trop\| (\nu) & = \theta_\trop(\nu) + c_{\trop,\rr}(\nu) \\
 & = c_{\trop,\rr}(\nu) \\
 & =  \frac{\ell}{2}B_2(\nu/\ell) - \frac{1}{12} \ell 
\end{split} 
\end{equation}
for $\nu \in [0,\ell]$.
 
\subsection{Explicit formulas}

Let $s$ be the global section of $L$ corresponding to the  theta function $\theta$ as in \eqref{fourier}. 
Applying Proposition~\ref{prop:ex_8.15} we find
\begin{equation} \label{eqn:with_B2}
 \begin{split} -\log \|s(x)\|_L & = c_{\trop,\rr}(\val(z)) - \log |\theta(z)| \\
   & = \frac{\ell}{2}B_2\left(\frac{\log|z|}{\log |q|}\right) - \log |\theta(z)| - \frac{1}{12} \ell 
\end{split} 
\end{equation}
for $x \in A^\an \setminus |D^\an| $ and $z \in \gg_\m^\an$ with $p(z) = x$.

Next, by Proposition~\ref{prop:f_trop_log_s}, we have 
\[  -\log\|s\|_L=\|\theta_\trop\| \]
as functions on $\varSigma = \rr/\ell\zz$. Let $\mu_H$ denote the Haar measure on $\varSigma$. Applying \eqref{eqn:theta_trop_abs} gives
\begin{equation} \label{eqn:integral}  \int_{A^\an} \log \|s\|_L \, \d \, \mu_H = - \int_0^\ell  \frac{\ell}{2}B_2(\nu/\ell)   \frac{\d \nu}{\ell} +  \frac{1}{12} \ell = \frac{1}{12} \ell \, . 
\end{equation}
 Combining  \eqref{eqn:with_B2} and \eqref{eqn:integral}  we find
 \[ \lambda'_D(x) = \frac{\ell}{2}B_2\left(\frac{\log|z|}{\log |q|}\right) - \log |\theta(z)| \]
for $x \in A^\an \setminus |D^\an| $ and $z \in \gg_\m(F)$ lifting $x$.

Let $\nn$ be the N\'eron model of $A$ over $R$. As is well known, the group of connected components $\varPhi_\nn \cong \zz/\ell\zz$ has a canonical (up to orientation) cyclic structure. One can obtain this cyclic structure either from the canonical embedding of $\varPhi_\nn \cong \zz/\ell\zz$ in the circle $\varSigma \cong \rr/\ell\zz$, or from embedding $\nn$ into the minimal regular model of $A$ over $R$, whose special fiber is a chain of $\mathbb{P}^1$'s. 

We pick one of the orientations of $\varPhi_\nn$ and accordingly label the elements of $\varPhi_\nn$ by the integers $i=0,\ldots,\ell-1$. Let $x \in A(F) \setminus |D|$. Applying Proposition~\ref{prop:normalized_classical} we find
 \[  \lambda'_D(x) = \mathbf{i}(x,D) + \frac{\ell}{2}B_2\left(i/\ell \right) \, ,  \]
if $\sp(x) = i \bmod \ell \zz$. This agrees with the outcome of calculations done in \cite[pp.\ 317--318]{neron} and in \cite[Chapitre~III.1.5]{mb}.

\section{Local-to-global formula for the canonical height} \label{sec:local-global}

Let $K$ be a number field. Let $A$ be an abelian variety over~$K$, and let $L$ be an ample line bundle on $A$.    The aim of this section is to give a version of N\'eron's local-to-global formula for the canonical height  $\h_L \colon A(K) \to \rr$ (see \cite[\S II.14]{neron},  \cite[Chapter~11]{lang}) in which {\em no unspecified constants} occur. 

Let $M_{K,0}$ be the set of non-archimedean places of $K$, let $M_{K,\infty}$ be the set of archimedean places of $K$, and set $M_K = M_{K,0} \sqcup M_{K,\infty}$. For each $v \in M_K$ let $Nv$ be the ``usual'' local factor defined as follows. If $v \in M_{K,0}$, then $Nv$ is the cardinality of the residue field at $v$. If $v \in M_{K,\infty}$, then $\log Nv=1$ if $v$ is real, and $\log Nv=2$ if $v$ is a pair of complex conjugate embeddings. 

Choose a rigidification of $L$.
 For each $v \in M_K$ let $\|\cdot\|_{L,v}$ denote the canonical metric on $L_v$ over $A_v^\an$ determined by the induced rigidification of $L_v$. For $v$ non-archimedean this is as discussed in \S\ref{sec:taut_metric}; for $v$ archimedean this is the unique smooth hermitian metric on $L_v$ which is compatible with the rigidification and has translation-invariant curvature form. 
 
Let $s \in H^0(A,L)$ be a non-zero global section of $L$, and write $D=\divisor_L(s)$. Let $x \in A(K) \setminus |D|$. We have the local-to-global formula
\begin{equation} \label{eq:local_global}
 \h_L(x) = -  \frac{1}{[K:\qq]}  \sum_{v \in M_K} \log \| s (x_v) \|_{L,v} \log Nv \, ,  
\end{equation}
where for each $v \in M_K$ we denote by $x_v$ the image of $x$ in $A(K_v)$ under the inclusion map $A(K) \to A(K_v)$. See, for instance, \cite[Corollary~9.5.14]{bg_heights}.

When $v$ is a non-archimedean place, we have an explicit formula for the local contribution $-\log \|s(x_v)\|_{L,v}$ by Theorem~\ref{thm:A}(ii). When $v$ is an archimedean place, one has a classical explicit formula for the local contribution $-\log \|s(x_v)\|_{L,v}$, in terms of \emph{normalized} complex theta functions for the divisor~$D$, see for instance \cite[Chapter~7]{neron} or \cite[Chapter~13]{lang}. 

This leads to the following result.
\begin{thm} \label{thm:B} Assume that $A$ has split semi-abelian reduction over $K$. Let $L$ be a rigidified ample line bundle over $A$. Let $s$ be a non-zero global section of $L$, and set $D = \divisor_L(s)$. For each $v \in M_K$ let a function 
\[ \lambda_{D,v} \colon A(K_v) \setminus |D_v| \to \rr \]
be given as follows.  

When $v$ is an archimedean place, let $T_{A,v}$ be the tangent space at the origin of the complex abelian variety $A_v^\an = A(\overline{K}_v)$, let $f_v \colon T_{A,v} \to \mathbb{C}$ be the normalized complex theta function associated to the section $s$ at $v$ and the induced rigidification of $L_v$, and let  $H_v \colon T_{A,v} \times T_{A,v} \to \mathbb{C}$ be the Riemann form of $L_v$. Let $x_v \in A(K_v) \setminus |D_v|$. If $z_v \in T_{A,v}$ is such that $\exp(z_v) = x_v$, we set
\begin{equation} \label{eqn:local_arch}  
\lambda_{D,v}(x_v) = - \log |f_v(z_v)| + \frac{\pi}{2}  H_v(z_v,z_v)   \, . 
\end{equation}
This does not depend on the choice of $z_v \in T_{A,v}$ lifting $x_v$.

When $v$ is a non-archimedean place, let $f_v$ be any theta function associated to the section~$s$ at~$v$, and let $\|f_{\trop,v}\| \colon \varSigma_v \to \rr$ denote the associated normalized tropicalized theta function, where $\varSigma_v \subset A_v^\an$ is the canonical skeleton. Let $\varPhi_v$ be the group of connected components of the N\'eron model of $A$ at $v$; then $\varPhi_v$ is canonically a finite subgroup of $\varSigma_v$. For every $x_v \in A(K_v) \setminus |D_v|$ we set
\begin{equation}  \label{eqn:local_non_arch}
\lambda_{D,v}(x_v) = \mathbf{i}_v(x_v,D_v) + \|f_{\trop,v}\|(x_{v,0}) \, . 
 \end{equation}
Here $x_{v,0}$ is the element of $\varPhi_v$ to which $x_v$ specializes, and $\mathbf{i}_v( x_v , D_v)$ denotes the intersection multiplicity of the Zariski closure of $x_v$ and the thickening of $D_v$, taken on the N\'eron model of $A$ at $v$.

Now let $x \in A(K) \setminus |D|$. Then we have the local-to-global formula
\[ \h_L(x) =  \frac{1}{[K:\qq]} \sum_{v \in M_K} \lambda_{D,v}(x_v) \log Nv \, ,  \]
where $\h_L$ is the canonical height associated to $L$, and where for each $v \in M_K$ we denote by $x_v$ the image of $x$ in $A(K_v)$ under the inclusion map $A(K) \to A(K_v)$.
\end{thm}

A similar result (with a similar proof) holds for abelian varieties defined over function fields. 

\begin{remark} Let $v \in M_K$ be non-archimedean, and let $\mu_{H,v}$ denote the Haar measure, with volume normalized to one, on the canonical skeleton $\varSigma_v \subset A_v^\an$. By pushforward we obtain a measure, also denoted $\mu_{H,v}$, on $A_v^\an$. From Definition~\ref{def:normalized} we recall the normalized canonical local height at $v$ associated to the divisor $D$,
\begin{equation} \label{eqn:def_our_normalized_bis}
 \lambda'_{D,v}(x_v) = -\log \|s(x_v)\|_{L,v} +  \int_{A_v^\an} \log \|s\|_{L,v} \, \d \, \mu_{H,v} \, , \quad x_v \in A_v^\an \setminus |D_v^\an| \, .
\end{equation}
When $v \in M_K$ is \emph{archimedean}, the very same formula defines a normalized canonical local height at $v$ associated to $D$, where now we interpret $A_v^\an$ as the set of complex valued points $A(\overline{K}_v)$ of $A$ and $\mu_{H,v}$ is the Haar measure on the complex torus $A(\overline{K}_v)$, with volume normalized to one.

Define the invariant
\[ \kappa'(A,D) = \frac{1}{[K:\qq]} \sum_{v \in M_K} \int_{A_v^\an} \log \|s\|_{L,v} \, \log Nv \, \d \, \mu_{H,v} \, , \]
which is readily checked to be a finite sum, independent of the choice of the section~$s$ by the product formula. Combining with \eqref{eq:local_global} we find the local-to-global formula
\[  \h_L(x) = - \kappa'(A,D) +  \frac{1}{[K:\qq]}  \sum_{v \in M_K} \lambda'_{D,v}(x_v) \log Nv  \]
for  $x \in A(K) \setminus |D|$, where $\h_L$ is the canonical height associated to $L$.

Assume that the divisor $D$ is \emph{integral}. Setting $g=\dim(A)$, it can be shown that
\[ \kappa'(A,D) = g \cdot \h_L(D) \, , \]
with $\h_L(D) \in \rr_{\ge 0}$ the N\'eron--Tate height \cites{bost_duke, bgs, gu-hohen, ph, zhsmall} of the subvariety $D \subset A$ with respect to the ample line bundle $L$. We arrive at the interesting formula
\[  \h_L(x) = - g \cdot \h_L(D) +  \frac{1}{[K:\qq]}  \sum_{v \in M_K} \lambda'_{D,v}(x_v) \log Nv  \]
for  $x \in A(K) \setminus |D|$. We intend to discuss some applications of this identity in future work.
\end{remark}

\end{document}